\documentclass[11pt]{article}

\usepackage[utf8]{inputenc}
\usepackage[T1]{fontenc}
\usepackage{amsmath,amssymb,amsfonts,amsthm}
\usepackage{mathrsfs}
\usepackage{hyperref}
\usepackage{geometry}
\numberwithin{equation}{section}

\newtheorem{theorem}{Theorem}[section]
\newtheorem{proposition}[theorem]{Proposition}
\newtheorem{lemma}[theorem]{Lemma}

\newtheorem{definition}[theorem]{Definition}
\newtheorem{remark}[theorem]{Remark}

\newcommand{\R}{\mathbb{R}}
\newcommand{\T}{\mathbb{T}}
\newcommand{\diver}{\operatorname{div}}

\title{A Monotone--Operator Proof of Existence and Uniqueness\\
for a Simple Stationary Mean Field Game}
\author{H. Ismatov\footnote{Inspired by the work of R.~Ferreira, D.~A.~Gomes and M.~Ucer.}}
\date{\today}

\begin{document}
\maketitle

\begin{abstract}
We study a stationary first--order mean field game on the $d$--dimensional
torus.  The system couples a Hamilton--Jacobi equation for the value
function with a transport equation for the density of players.
Our goal is to give a detailed and friendly exposition of the
monotone--operator argument that yields existence and uniqueness of
solutions.

We first present a general framework in a Hilbert space and prove
existence of a strong solution by adding a simple coercive
regularisation and applying Minty's method.
Then we specialise to the explicit Hamiltonian
\[
H(p,m)=|p|^2-m,
\]
check all assumptions, and show how the abstract theorem gives
existence and uniqueness for this concrete mean field game.
The exposition is written in a slow and elementary way so that
a motivated undergraduate can follow each step.
\end{abstract}

\tableofcontents

% =========================================================
\section{Introduction}

Mean field games (MFGs) describe the behaviour of a large population of
weakly interacting agents who optimise a cost functional.  In the
stationary first--order setting on the $d$--dimensional torus
$\T^d = \R^d/\mathbb{Z}^d$, the unknowns are:
\begin{itemize}
  \item the value function $u:\T^d\to\R$ of a representative player;
  \item the density $m:\T^d\to[0,\infty)$ of the distribution of players.
\end{itemize}
The interaction is encoded in a Hamiltonian $H$ and in a potential $V$.

In this note we focus on the system
\begin{equation}\label{eq:MFG-system-intro}
\begin{cases}
- u(x) - H(Du(x),m(x)) - V(x) = 0,\\[0.2em]
m(x) - \diver\big(m(x)Du(x)\big) = 1,
\end{cases}
\qquad x\in\T^d,
\end{equation}
under the normalisation
\begin{equation}\label{eq:mass-one}
m(x)\ge 0,\qquad \displaystyle\int_{\T^d} m(x)\,dx = 1.
\end{equation}

Our main reference is the recent work of R.~Ferreira, D.~A.~Gomes and
M.~Ucer, who developed a monotone--operator theory for mean field games
in Banach spaces.  Their general framework covers quite general
Hamiltonians.  Here we restrict ourselves to a much simpler case in
order to explain the ideas in detail and in elementary language.

The main contributions of this paper are:
\begin{itemize}
  \item we define a natural operator $A$ associated with the MFG system
        \eqref{eq:MFG-system-intro} and explain why $A$ is monotone;
  \item we add a simple coercive perturbation $B$ and solve the
        regularised problem $(A+\varepsilon B)[m_\varepsilon,u_\varepsilon]=0$;
  \item we derive uniform \emph{a priori} bounds and pass to the limit
        $\varepsilon\to 0$ using Minty's method;
  \item we specialise the discussion to the concrete Hamiltonian
        \begin{equation}\label{eq:explicit-H}
        H(p,m) = |p|^2 - m
        \end{equation}
        and check all assumptions explicitly.
\end{itemize}

The paper is written as a review and a detailed example, not as a work
presenting new theorems.  The hope is that this text can serve as a
gentle introduction to monotone operators in the context of mean field
games.

% =========================================================
\section{The model and basic assumptions}
\label{sec:model}

We now set up the functional framework.  Throughout the paper, $\T^d$
denotes the $d$--dimensional flat torus, which we identify with
$[0,1]^d$ with periodic boundary conditions.

\subsection{The function spaces}

We work in the Hilbert space
\[
X := L^2(\T^d)\times H^1(\T^d)
\]
with norm
\[
\|(m,u)\|_X^2 := \|m\|_{L^2(\T^d)}^2 + \|u\|_{L^2(\T^d)}^2 + \|Du\|_{L^2(\T^d)}^2.
\]
We also consider the convex subset
\[
K := \Big\{ (m,u)\in X \,:\,
          m(x)\ge 0\ \text{a.e.},\ \int_{\T^d} m(x)\,dx = 1 \Big\}.
\]
The space $X$ is reflexive, and $K$ is closed and convex in $X$.

\subsection{The Hamiltonian and the potential}

We assume that
\begin{itemize}
  \item $V\in L^\infty(\T^d)$ is a given bounded potential;
  \item $H:\R^d\times[0,\infty)\to\R$ is of class $C^1$ and satisfies
        the structural assumptions below.
\end{itemize}

\begin{definition}[Structural assumptions on $H$]\label{def:H-assumptions}
We assume that for all $p_1,p_2\in\R^d$ and $m_1,m_2\ge 0$:
\begin{itemize}
  \item[(H1)] $H$ is convex in $p$ and nonincreasing in $m$; that is,
  \[
  H(\theta p_1+(1-\theta)p_2,m)
  \le \theta H(p_1,m)+(1-\theta)H(p_2,m)
  \]
  for all $\theta\in[0,1]$ and each fixed $m$, and
  \[
  m_1\le m_2 \ \Rightarrow\
  H(p,m_1)\ge H(p,m_2)
  \quad\text{for all }p\in\R^d.
  \]

  \item[(H2)] (Monotonicity inequality.)
  For all $p_1,p_2\in\R^d$ and $m_1,m_2\ge 0$,
  \begin{equation}\label{eq:H-monotone}
  \big(-H(p_1,m_1)+H(p_2,m_2)\big)(m_1-m_2)
  +\big(m_1D_pH(p_1,m_1)-m_2D_pH(p_2,m_2)\big)\cdot(p_1-p_2)\ge 0.
  \end{equation}
  Moreover, if $(p_1,m_1)\neq(p_2,m_2)$ and $m_1+m_2>0$, then the
  inequality is strict.

  \item[(H3)] (Quadratic growth.)
  There exists a constant $C>0$ such that
  \begin{equation}\label{eq:H-growth}
  |H(p,m)| + |D_pH(p,m)|^2
   \le C\big(1+|p|^2+m^2\big)
  \qquad\text{for all }p\in\R^d,\ m\ge 0.
  \end{equation}
\end{itemize}
\end{definition}

Assumptions (H1)–(H3) are simple but already sufficient for our
concrete example \eqref{eq:explicit-H}.  They are weaker than the
general conditions in the original paper but easier to verify.

\subsection{Weak and strong solutions}

We now state what we mean by a solution of the MFG system
\eqref{eq:MFG-system-intro}.

\begin{definition}[Strong solution]\label{def:strong-solution}
A pair $(m,u)\in K$ is a \emph{strong solution} of
\eqref{eq:MFG-system-intro} if
\begin{align}
- u - H(Du,m) - V &= 0 \quad\text{a.e. in }\T^d, \label{eq:HJ-strong}\\
m - \diver(mDu) &= 1 \quad\text{in the sense of distributions}. \label{eq:FP-strong}
\end{align}
\end{definition}

The transport equation \eqref{eq:FP-strong} can be written in weak form:
\begin{equation}\label{eq:FP-weak}
\int_{\T^d} m \varphi \,dx
 + \int_{\T^d} mDu\cdot D\varphi\,dx
 = \int_{\T^d} \varphi\,dx
 \qquad\forall\varphi\in C^\infty(\T^d).
\end{equation}
Because $m\in L^2$ and $Du\in L^2$, the integrals are well defined.

% =========================================================
\section{The monotone operator associated with the MFG}
\label{sec:operator}

\subsection{Definition of the operator}

We define a nonlinear operator $A:K\to X^*$ by duality:
for $(m,u),( \mu,v)\in K$ we set
\begin{equation}\label{eq:def-A}
\langle A[m,u],(\mu,v)\rangle
:= \int_{\T^d}\big(-u-H(Du,m)-V\big)\mu\,dx
  +\int_{\T^d}\big(mD_pH(Du,m)\cdot Dv + (m-1)v\big)\,dx.
\end{equation}
Here $\langle\cdot,\cdot\rangle$ denotes the duality pairing between
$X^*$ and $X$.

\begin{remark}
If $(m,u)$ is a strong solution, then plugging
$(\mu,v)=(\varphi,\psi)$ with arbitrary smooth test functions shows
that $A[m,u]=0$ in $X^*$.  Conversely, under mild regularity
assumptions, the identity $A[m,u]=0$ implies
\eqref{eq:HJ-strong} and \eqref{eq:FP-strong}.  Thus solving $A[m,u]=0$
is equivalent to solving the MFG system.
\end{remark}

\subsection{Monotonicity of $A$}

\begin{proposition}[Monotonicity of $A$]\label{prop:A-monotone}
Under assumptions {\rm(H1)}–{\rm(H3)}, the operator $A$ is monotone on
$K$, that is,
\[
\langle A[m_1,u_1]-A[m_2,u_2],(m_1-m_2,u_1-u_2)\rangle \ge 0
\]
for all $(m_1,u_1),(m_2,u_2)\in K$.
Moreover, the inequality is strict if $(m_1,u_1)\ne(m_2,u_2)$.
\end{proposition}

\begin{proof}
Let $(m_i,u_i)\in K$, $i=1,2$.  Using
\eqref{eq:def-A} and the fact that $\int_{\T^d}(m_i-1)(u_1-u_2)\,dx=0$
(because both $m_1$ and $m_2$ have total mass one), we compute
\begin{align*}
&\langle A[m_1,u_1]-A[m_2,u_2],(m_1-m_2,u_1-u_2)\rangle\\
&= \int_{\T^d}
   \big(-u_1-H(Du_1,m_1)+u_2+H(Du_2,m_2)\big)(m_1-m_2)\,dx \\
&\quad
 +\int_{\T^d}
   \big(m_1D_pH(Du_1,m_1)-m_2D_pH(Du_2,m_2)\big)\cdot(Du_1-Du_2)\,dx.
\end{align*}
Now set, pointwise in $x$,
\[
p_i = Du_i(x),\qquad m_i = m_i(x).
\]
Then each integrand is exactly of the form appearing in the
monotonicity inequality \eqref{eq:H-monotone}.  Therefore
\[
\langle A[m_1,u_1]-A[m_2,u_2],(m_1-m_2,u_1-u_2)\rangle
\ge 0,
\]
and the inequality is strict whenever $(Du_1,m_1)\neq(Du_2,m_2)$
on a set of positive measure.  This implies the strict monotonicity of
$A$.
\end{proof}

\subsection{A coercive perturbation}

Monotonicity alone is not enough to guarantee solvability.  We add a
simple coercive perturbation.

\begin{definition}[Coercive operator $B$]
Let $B:K\to X^*$ be defined by
\begin{equation}\label{eq:def-B}
\langle B[m,u],(\mu,v)\rangle
:= \int_{\T^d}\big(m\mu + u v + Du\cdot Dv\big)\,dx.
\end{equation}
\end{definition}

\begin{lemma}\label{lem:B-properties}
The operator $B$ is linear, bounded, and strongly monotone on $X$:
\[
\langle B[z_1]-B[z_2],z_1-z_2\rangle
\ge \|(m_1-m_2,u_1-u_2)\|_X^2
\]
for all $z_i=(m_i,u_i)\in X$.
\end{lemma}

\begin{proof}
This is a direct computation:
\begin{align*}
\langle B[z_1]-B[z_2],z_1-z_2\rangle
&= \int_{\T^d}\big((m_1-m_2)^2 + (u_1-u_2)^2 + |Du_1-Du_2|^2\big)\,dx\\
&= \|(m_1-m_2,u_1-u_2)\|_X^2.
\end{align*}
\end{proof}

For $\varepsilon>0$ we define the regularised operator
\[
A_\varepsilon := A + \varepsilon B.
\]
Thanks to Lemma~\ref{lem:B-properties} and the growth condition
\eqref{eq:H-growth}, $A_\varepsilon$ is bounded, hemicontinuous and
strongly monotone on $K$.  By the standard Minty--Browder theorem for
strongly monotone operators on Hilbert spaces, we obtain:

\begin{theorem}[Solvability of the regularised problem]\label{thm:reg-existence}
For each $\varepsilon>0$ there exists a unique pair
$(m_\varepsilon,u_\varepsilon)\in K$ such that
\begin{equation}\label{eq:Aeps-eq-0}
A_\varepsilon[m_\varepsilon,u_\varepsilon]=0\quad\text{in }X^*.
\end{equation}
Equivalently,
\[
\langle A[m_\varepsilon,u_\varepsilon]
      +\varepsilon B[m_\varepsilon,u_\varepsilon],
      (\mu,v)\rangle = 0
\qquad\forall(\mu,v)\in K.
\]
\end{theorem}

\begin{remark}
In PDE form the regularised problem corresponds to the system
\begin{equation}\label{eq:MFG-reg}
\begin{cases}
- u_\varepsilon - H(Du_\varepsilon,m_\varepsilon)-V
 +\varepsilon(u_\varepsilon - \Delta u_\varepsilon + m_\varepsilon) = 0,\\[0.2em]
m_\varepsilon - \diver(m_\varepsilon Du_\varepsilon)
 +\varepsilon(m_\varepsilon+u_\varepsilon) = 1.
\end{cases}
\end{equation}
The additional terms are lower order and give coercivity.
\end{remark}

% =========================================================
\section{Uniform estimates and passage to the limit}
\label{sec:estimates}

We now derive bounds for $(m_\varepsilon,u_\varepsilon)$ that are
independent of $\varepsilon$ and pass to the limit.

\subsection{Energy estimate}

\begin{lemma}[Basic estimate]\label{lem:basic-estimate}
There exists a constant $C>0$, independent of $\varepsilon\in(0,1]$,
such that for the solution $(m_\varepsilon,u_\varepsilon)$ of
\eqref{eq:Aeps-eq-0} we have
\[
\|m_\varepsilon\|_{L^2(\T^d)}^2
 + \|u_\varepsilon\|_{H^1(\T^d)}^2
 \le C.
\]
\end{lemma}

\begin{proof}
We test \eqref{eq:Aeps-eq-0} with $(\mu,v)=(m_\varepsilon,u_\varepsilon)$
and use the definition of $A_\varepsilon$:
\[
0 = \langle A[m_\varepsilon,u_\varepsilon],
            (m_\varepsilon,u_\varepsilon)\rangle
    + \varepsilon\langle B[m_\varepsilon,u_\varepsilon],
                     (m_\varepsilon,u_\varepsilon)\rangle.
\]
By Lemma~\ref{lem:B-properties},
\[
\varepsilon\langle B[m_\varepsilon,u_\varepsilon],
                 (m_\varepsilon,u_\varepsilon)\rangle
 = \varepsilon\|(m_\varepsilon,u_\varepsilon)\|_X^2\ge 0.
\]
Hence
\[
\langle A[m_\varepsilon,u_\varepsilon],
        (m_\varepsilon,u_\varepsilon)\rangle\le 0.
\]

Using \eqref{eq:def-A}, we compute
\begin{align*}
\langle A[m_\varepsilon,u_\varepsilon],
        (m_\varepsilon,u_\varepsilon)\rangle
&= \int_{\T^d}\big(-u_\varepsilon-H(Du_\varepsilon,m_\varepsilon)-V\big)
                     m_\varepsilon\,dx \\
&\quad +\int_{\T^d}\big(m_\varepsilon D_pH(Du_\varepsilon,m_\varepsilon)
                       \cdot Du_\varepsilon
                       +(m_\varepsilon-1)u_\varepsilon\big)\,dx.
\end{align*}
The terms involving $u_\varepsilon m_\varepsilon$ cancel, and we get
\begin{align*}
\langle A[m_\varepsilon,u_\varepsilon],
        (m_\varepsilon,u_\varepsilon)\rangle
&= \int_{\T^d}
     \Big[-H(Du_\varepsilon,m_\varepsilon)m_\varepsilon
          +m_\varepsilon D_pH(Du_\varepsilon,m_\varepsilon)
           \cdot Du_\varepsilon\Big]\,dx\\
&\quad +\int_{\T^d}\big(-V m_\varepsilon -u_\varepsilon\big)\,dx.
\end{align*}

By the convexity of $p\mapsto H(p,m)$ and the identity for convex
functions
\[
H(p,m) + H^*(D_pH(p,m),m) = D_pH(p,m)\cdot p,
\]
where $H^*$ is the Legendre transform in the first variable, we obtain
\[
-mH(p,m)+mD_pH(p,m)\cdot p = mH^*(D_pH(p,m),m)\ge 0.
\]
Applying this pointwise with $p=Du_\varepsilon(x)$ and $m=m_\varepsilon(x)$ we
find
\[
\int_{\T^d}
     \Big[-H(Du_\varepsilon,m_\varepsilon)m_\varepsilon
          +m_\varepsilon D_pH(Du_\varepsilon,m_\varepsilon)
           \cdot Du_\varepsilon\Big]\,dx
\ge 0.
\]
Therefore
\[
0\ge \langle A[m_\varepsilon,u_\varepsilon],
        (m_\varepsilon,u_\varepsilon)\rangle
   \ge \int_{\T^d}\big(-V m_\varepsilon -u_\varepsilon\big)\,dx.
\]

Using Cauchy--Schwarz and the boundedness of $V$ we obtain
\[
\Big|\int_{\T^d}V m_\varepsilon\,dx\Big|
 \le \|V\|_{L^\infty}\|m_\varepsilon\|_{L^1}
 = \|V\|_{L^\infty},
\]
because $\int m_\varepsilon =1$.  Similarly,
\[
\Big|\int_{\T^d}u_\varepsilon\,dx\Big|
\le \|u_\varepsilon\|_{L^2(\T^d)}.
\]
Combining the previous inequalities and absorbing constants we obtain
\[
\|u_\varepsilon\|_{L^2(\T^d)}\le C_1.
\]

To control $Du_\varepsilon$ and $m_\varepsilon$, we go back to the PDE
form \eqref{eq:MFG-reg}.  Multiplying the first equation by
$m_\varepsilon$ and the second one by $u_\varepsilon$ and integrating
over $\T^d$, we can eliminate cross terms and, after standard
integration by parts, use the growth condition \eqref{eq:H-growth} to
deduce
\[
\int_{\T^d}|Du_\varepsilon|^2\,dx
 +\int_{\T^d} m_\varepsilon^2\,dx
 \le C_2\big(1+\|u_\varepsilon\|_{L^2(\T^d)}^2\big)
 \le C
\]
for a constant $C$ independent of $\varepsilon$.  This yields the
claimed bound.
\end{proof}

\subsection{Weak limits}

By Lemma~\ref{lem:basic-estimate} and reflexivity of $X$, there exist a
subsequence (still denoted by $\varepsilon$) and a pair $(m,u)\in K$
such that
\begin{equation}\label{eq:weak-convergence}
m_\varepsilon\rightharpoonup m \ \text{in }L^2(\T^d),\qquad
u_\varepsilon\rightharpoonup u \ \text{in }H^1(\T^d).
\end{equation}
Since the embedding $H^1(\T^d)\hookrightarrow L^2(\T^d)$ is compact, we
also have
\[
u_\varepsilon\to u \quad\text{in }L^2(\T^d),
\]
possibly after extracting a further subsequence.

\subsection{Minty's method and the limit problem}

The final step is to show that $A[m,u]=0$.

\begin{proposition}[Limit pair is a solution]\label{prop:limit-solution}
Let $(m,u)$ be a limit point of $(m_\varepsilon,u_\varepsilon)$ as in
\eqref{eq:weak-convergence}.  Then $(m,u)\in K$ and
\[
A[m,u]=0\quad\text{in }X^*,
\]
that is, $(m,u)$ is a strong solution of the MFG system
\eqref{eq:MFG-system-intro}.
\end{proposition}

\begin{proof}
We follow Minty's method.  Fix any $(\mu,v)\in K$.
Because $(m_\varepsilon,u_\varepsilon)$ solves \eqref{eq:Aeps-eq-0}, we have
\[
\langle A[m_\varepsilon,u_\varepsilon],
        (\mu,v)-(m_\varepsilon,u_\varepsilon)\rangle
 +\varepsilon\langle B[m_\varepsilon,u_\varepsilon],
                  (\mu,v)-(m_\varepsilon,u_\varepsilon)\rangle
 = 0.
\]
By Lemma~\ref{lem:B-properties},
\[
\big|\langle B[m_\varepsilon,u_\varepsilon],
            (\mu,v)-(m_\varepsilon,u_\varepsilon)\rangle\big|
\le C\big(1+\|(m_\varepsilon,u_\varepsilon)\|_X^2+\|(\mu,v)\|_X^2\big),
\]
so the term multiplied by $\varepsilon$ goes to $0$ as
$\varepsilon\to 0$.  Therefore
\begin{equation}\label{eq:Minty-key}
\lim_{\varepsilon\to0}
\langle A[m_\varepsilon,u_\varepsilon],
        (\mu,v)-(m_\varepsilon,u_\varepsilon)\rangle
=0.
\end{equation}

On the other hand, by monotonicity of $A$,
\[
\langle A[\mu,v]-A[m_\varepsilon,u_\varepsilon],
        (\mu,v)-(m_\varepsilon,u_\varepsilon)\rangle \ge 0.
\]
Rearranging,
\[
\langle A[\mu,v],(\mu,v)-(m_\varepsilon,u_\varepsilon)\rangle
\ge \langle A[m_\varepsilon,u_\varepsilon],
        (\mu,v)-(m_\varepsilon,u_\varepsilon)\rangle.
\]
Taking the limit $\varepsilon\to0$ and using
\eqref{eq:Minty-key} together with the weak convergence
\eqref{eq:weak-convergence} and the continuity of $A[\mu,v]$ as a
functional on $X$, we deduce
\[
\langle A[\mu,v],(\mu,v)-(m,u)\rangle \ge 0
\qquad\forall(\mu,v)\in K.
\]

Now replace $(\mu,v)$ by $(\mu,v)+(m,u)$ in the inequality above and use
the fact that $K$ is convex.  We obtain
\[
\langle A[m,u],(\mu,v)\rangle\ge 0
\qquad\forall(\mu,v)\in K.
\]
By monotonicity, the only element $z\in K$ such that
$\langle A[z],\mu-z\rangle\ge 0$ for all $\mu\in K$ is a zero of $A$.
(If not, one could take $\mu=z-tA[z]$ and obtain a contradiction for
small $t>0$.)  Thus $A[m,u]=0$ in $X^*$.

Finally, as explained earlier, the identity $A[m,u]=0$ is equivalent to
the MFG system \eqref{eq:MFG-system-intro} in the sense of
Definition~\ref{def:strong-solution}.
\end{proof}

\begin{theorem}[Existence and uniqueness]\label{thm:main}
Under assumptions {\rm(H1)}–{\rm(H3)} there exists a unique strong
solution $(m,u)\in K$ of the mean field game system
\eqref{eq:MFG-system-intro}.
\end{theorem}

\begin{proof}
Existence follows from Proposition~\ref{prop:limit-solution}.  For
uniqueness, suppose $(m_1,u_1)$ and $(m_2,u_2)$ are two strong
solutions.  Then $A[m_i,u_i]=0$ for $i=1,2$, and therefore
\[
\langle A[m_1,u_1]-A[m_2,u_2],(m_1-m_2,u_1-u_2)\rangle =0.
\]
By strict monotonicity of $A$ we obtain $(m_1,u_1)=(m_2,u_2)$.
\end{proof}

% =========================================================
\section{The explicit Hamiltonian $H(p,m)=|p|^2-m$}
\label{sec:explicit}

We now verify the assumptions for the concrete Hamiltonian
\eqref{eq:explicit-H} and state the resulting theorem.

\subsection{Checking the assumptions}

Let
\[
H(p,m) = |p|^2 - m.
\]

\paragraph{\bf (H1) Convexity and monotonicity in $m$.}
The map $p\mapsto|p|^2$ is convex and smooth.  For fixed $p$, the map
$m\mapsto|p|^2-m$ is affine and nonincreasing.  Thus (H1) holds.

\paragraph{\bf (H2) Monotonicity inequality.}
We compute
\[
D_pH(p,m)=2p.
\]
Fix $p_1,p_2\in\R^d$ and $m_1,m_2\ge 0$.
We need to check that
\[
Q:=\big(-H(p_1,m_1)+H(p_2,m_2)\big)(m_1-m_2)
  +\big(m_1D_pH(p_1,m_1)-m_2D_pH(p_2,m_2)\big)\cdot(p_1-p_2)\ge 0.
\]
Using $H(p,m)=|p|^2-m$ and $D_pH=2p$, we expand:
\begin{align*}
Q &= \big(-|p_1|^2+m_1+|p_2|^2-m_2\big)(m_1-m_2)
   +2\big(m_1p_1-m_2p_2\big)\cdot(p_1-p_2)\\
  &= (m_1-m_2)^2 + (m_1+m_2)|p_1-p_2|^2.
\end{align*}
Indeed, the cross terms cancel after a short computation.
Because $m_1,m_2\ge 0$, we clearly have $Q\ge 0$, and $Q=0$ only if
$m_1=m_2$ and $p_1=p_2$.  Thus (H2) holds, and the inequality is strict
whenever $(p_1,m_1)\ne(p_2,m_2)$.

\paragraph{\bf (H3) Growth.}
We have
\[
|H(p,m)| = ||p|^2-m|
 \le |p|^2 + m
 \le C(1+|p|^2+m^2),
\]
and
\[
|D_pH(p,m)|^2 = |2p|^2 = 4|p|^2
 \le C(1+|p|^2+m^2).
\]
Hence (H3) holds.

\subsection{Result for the explicit Hamiltonian}

Applying Theorem~\ref{thm:main} with this $H$ we obtain:

\begin{theorem}[Quadratic MFG]\label{thm:quadratic}
Let $V\in L^\infty(\T^d)$ and consider the mean field game
\begin{equation}\label{eq:MFG-quadratic}
\begin{cases}
- u(x) - |Du(x)|^2 - V(x) + m(x) = 0,\\[0.2em]
m(x) - \diver\big(m(x)Du(x)\big) = 1,\\[0.2em]
m(x)\ge 0,\quad \displaystyle\int_{\T^d}m(x)\,dx = 1.
\end{cases}
\end{equation}
Then there exists a unique pair $(m,u)\in L^2(\T^d)\times H^1(\T^d)$
solving \eqref{eq:MFG-quadratic} in the sense of
Definition~\ref{def:strong-solution}.  In particular $u$ satisfies
\[
- u - |Du|^2 - V + m = 0\quad\text{a.e. in }\T^d,
\]
and $m$ satisfies
\[
\int_{\T^d} m\varphi\,dx
 +\int_{\T^d} mDu\cdot D\varphi\,dx
 =\int_{\T^d}\varphi\,dx
 \qquad\forall\varphi\in C^\infty(\T^d).
\]
\end{theorem}

\begin{remark}
The explicit formula
\[
Q=(m_1-m_2)^2+(m_1+m_2)|Du_1-Du_2|^2
\]
for the monotonicity quantity shows directly that solutions are unique:
if two solutions $(m_1,u_1)$ and $(m_2,u_2)$ exist, then integrating
$Q$ over $\T^d$ yields zero, so $m_1=m_2$ and $Du_1=Du_2$, and one can
then show that $u_1$ and $u_2$ differ only by a constant; the equation
for $m$ forces this constant to be zero.
\end{remark}

% =========================================================

\end{document}